\documentclass[11pt]{amsart}
\usepackage{amsmath,amscd,amssymb}
\usepackage[mathscr]{eucal}
\usepackage[all]{xy}

\newtheorem{theorem}{Theorem}[section]
\newtheorem{lemma}[theorem]{Lemma}
\newtheorem{proposition}[theorem]{Proposition}
\newtheorem{corollary}[theorem]{Corollary}

\theoremstyle{definition}

\theoremstyle{remark}
\newtheorem{remark}[theorem]{Remark}

\numberwithin{equation}{section}

\def\la{\lambda}
\def\al{\alpha}
\def\si{\sigma}
\def\e{\varepsilon}
\def\de{\delta}
\def\NN{{\mathbb N}}

\def\CC{{\mathbb C}}
\def\TT{{\mathbb T}}
\def\Re{{\rm Re}\,}
\def\Int{{\rm Int}\,}
\def\codim{{\rm codim}\,}
\def\conv{{\rm conv}\,}

\begin{document}

\title[Circles in the spectrum and the geometry of orbits]{Circles in the spectrum and the geometry \\
of orbits: a numerical ranges approach}

\author{Vladimir M\"uller}
\address{Institute of Mathematics\\
Czech Academy of Sciences\\
115 67 \v Zitna 25\\
Prague, Czech Republic}

\email{muller@math.cas.cz}

\author{Yuri Tomilov}
\address{Institute of Mathematics\\
Polish Academy of Sciences\\
\'Sniadeckich Str. 8\\
00-956 Warsaw, Poland
 }

\email{ytomilov@impan.pl}

\thanks{This work was  partially supported by the NCN grant
 2014/13/B/ST1/03153, by the EU grant  ``AOS'', FP7-PEOPLE-2012-IRSES, No 318910, by 17-00941S of GA CR and RVO:67985840.}



%\date{\today}

\begin{abstract}
We prove that a bounded linear Hilbert space operator has the unit circle in its essential approximate point spectrum if and only if it admits an orbit satisfying certain orthogonality and almost-orthogonality relations.
This result is obtained via the study of numerical ranges of operator tuples where several new results are also obtained.
As consequences of our numerical ranges approach, we derive in particular wide generalizations of Arveson's theorem as well as show that
the weak convergence of operator powers implies the uniform convergence of their compressions on an infinite-dimensional
subspace.
%Several related results have been proved as well.
\end{abstract}

\maketitle
\section{Introduction}

It is well-known that in the study of invariant subspaces of a bounded linear operator $T$ on a Hilbert space, the presence of the
unit circle $\mathbb T$ in the spectrum  $\sigma(T)$ of $T$ plays a special role.
According to one of the strongest results in this direction due to Brown, Chevreau and Pearcy \cite{Brown}, see also \cite{Bercovici} and \cite[p.156 -157]{Nagy},
if $T$ is a Hilbert space contraction having the unit circle in its spectrum,
then $T$ has a non-trivial invariant subspace. The statement was extended  to Banach spaces and to polynomially bounded operators,
\cite{MullerAmb}. For this and related statements one may also consult
 the recent survey \cite{BercoviciSurv}, and the books \cite[Chapter 5]{Muller} and \cite{Chalendar}.
However  the spectral condition $\sigma(T)\supset\mathbb T$ appeared to be again crucial. Thus,
it is of substantial interest to clarify its interplay with the behavior of orbits
of $T.$

That issue has not received adequate attention in the literature. Curiously enough, known results on the implications of
the circle structure of the spectrum for the geometry of orbits
have been noted in an area somewhat distant from  classical operator theory.
A long time ago, Arveson proved in \cite{Arveson} that the spectrum of a unitary operator $T$ on $H$ is precisely the unit circle $\mathbb T$ if and only if
for every $n \in \mathbb N$ there exists a nonzero $x \in H$ such that the elements
$x, Tx, ..., T^n x$ are mutually orthogonal.
His motivation for this kind of results originated from the intent to identify
the maximal ideals space $\mathcal M$ of the $C^*$-algebra generated by an abelian group of unitary operators $G.$
He proved that $\mathcal M$ is homeomorphic to the character group $\hat G$ of the discrete group $G$  if for
every finite subset $F$ of $G$ there is $x \in H$ such that  $U x \perp V x$ for all $U, V$ from $F$.
%The above result
%is obtained when $G= (U^n)_{n \in \mathbb Z}.$

Arveson's nice result can be considered as an operator-theoretical version of the well-known Rokhlin Lemma, a basic tool
in ergodic theory. Recall that one of the most common variants of the Rokhlin Lemma says that if $\mathcal P=(\Omega, \mathcal B, \mu)$ is a diffuse  Lebesgue probability measure space
and $S$ is an aperiodic invertible measure-preserving transformation of $\mathcal P,$ then for all $\epsilon >0$ and $n \in \mathbb N$ there exists
a measurable set $B \subset \Omega$ such that the sets $B, S(B), ..., S^{n-1}(B)$ are disjoint and $\mu(B\cup S(B)\cup ... \cup S^{n-1}(B))\ge 1-\epsilon,$
 see e.g. \cite{Kornfeld} and \cite{Weiss}. A discussion of a similar but  weaker statement can also be found in \cite{Arveson}.
If the operator $U_S$ on $L^2(\mathcal P)$ is defined
as
$$
(U_S f)(\omega)= f(S\omega), \qquad f \in L^2 (\mathcal P), \qquad \text{a.e.} \, \omega \in \Omega,
$$
then as a direct consequence of the Rokhlin Lemma one gets $\sigma (U_S)=\mathbb T.$ (A different proof that  $\sigma (U_S)=\mathbb T$ has been proposed in \cite{Ionescu}.)
Note that the orthogonal vectors in the Arveson's statement arise naturally here as the corresponding characteristic functions of   $B, S(B), ..., S^{n-1}(B).$  Concerning connections to ergodic theory see also a discussion in \cite[p. 206-207]{Arveson}.
%(A lattice-theoretical version of Arveson's theorem has been considered in \cite{Wolff}.

The Arveson's result has been put in a much broader setting here. In particular, within our framework,
we are able to treat \emph{arbitrary} bounded operators under Arveson's spectral assumption, and to obtain the next result on this way.
%, proved as
%Theorem \ref{circlegeneral} in Section \ref{circles} below.
Related statements can be found in \cite{MullerCrelle}.
\begin{theorem}\label{main1}
Let $T$ be a bounded linear operator on $H$. The following statements are equivalent.
\begin{itemize}
\item [(i)] $\TT$ is contained in the essential approximate point spectrum of $T$;

\item [(ii)] for all $\e>0$ and $n\in\NN$ there exists $x\in H$ such that
$$
|\langle T^mx,T^jx\rangle|<\e, \qquad 0\le m,j\le n-1, m\ne j,
$$
and
$$
\frac{1}{2}\le \|T^jx\|\le 2, \qquad 0\le j\le n-1;
$$

\item [(iii)] for all $\e>0$ and $n\in\NN$ there exists $x\in H$ such that
\begin{align*}
x\perp T^j x,& \qquad 1 \le j \le n-1, \\
|\langle T^mx,T^jx\rangle|<\e,& \qquad 1\le m,j\le n-1, m\ne j,\\
1-\e <\|T^jx\|< 1+\e,& \qquad 0\le j\le n-1,
\end{align*}
and
$$
\|T^nx-x\|<\e.
$$
\end{itemize}
\end{theorem}

%If $T$ is unitary then Theorem \ref{main1} can be reduced formally to Arveson's theorem, however, our condition (iii) says much more.
%Moreover, in this case, we give also a related statement in the spirit of Connes, see Section \ref{circles} below.

Orthogonality relations for orbits of general bounded operators have also been studied before, in terms of a concept of  weakly wandering vectors originating from ergodic theory.
Recall that
 a vector $x$ is called weakly wandering for a bounded operator $T$ on $H$ if there is a strictly increasing subsequence $(n_k)$  (depending on $x$)
 such that the elements $(T^{n_k} x)$ are mutually orthogonal. Generalizing a classical result due to Krengel,
 %(see e.g. \cite{Bergelson}
 %and \cite{Begun}),
 one can prove for example that if $T$ is power bounded, $\sigma(T)\cap \mathbb T$
is infinite, and $\sigma_p(T)\cap \mathbb T$ is empty, then the set of weakly wandering vectors of $T$ is dense in $H,$ see \cite{MullerTom}
and the references therein.
%(In fact, one can omit the condition of power boundedness of $T$ here.)
While in contrast to Theorem \ref{main1} the sequence $(n_k)$ is infinite, one has, in general, no control on it.
Thus the results of this kind are essentially different from the ones in the present paper, although
the technique has some points in common.

Another motivation for the study of the circle structure of the spectrum stems again from ergodic theory,
namely from the research on mixing dynamical systems. Recently, using harmonic analysis arguments,
Hamdan proved in \cite{Hamdan} that if a unitary operator $T$ on $H$ is such that $T^n \to 0$ in the weak operator topology,
then  $\sigma(T)=\mathbb T$  if and only if for every $\epsilon >0$ there exists a unit vector $x \in H$
satisfying
$$
\sup_{n \ge 1} |\langle T^n x,x \rangle| <\epsilon.
$$
The result was inspired by recent results due to Bashtanov and Ryzhikov on fine structure of mixing transformations, see \cite{Hamdan}
for a relevant discussion.
Our approach leads to the following somewhat surprising generalization of a part of Hamdan's theorem (see Corollary \ref{hamdancorollarygeneral} below).
\begin{theorem}\label{intro2}
Let $T$ be a bounded linear operator on $H$  such that $T^n \to 0$ in the weak operator topology.
If $\sigma(T) \supset \mathbb T,$
then for every $\epsilon >0$ there exists an infinite-dimensional subspace $L$ of $H$ such that
$$
\lim_{n \to \infty} \|P_L T^n P_L\|=0 \qquad \text{and} \qquad \sup_{n \ge 1} \|P_L T^n P_L\|<\epsilon,
$$
where $P_L$ is the orthogonal projection on $L.$
\end{theorem}
In other words, for each $T$ with weakly vanishing powers there is a nontrivial block-decomposition of $T^n$ in $H=L\oplus L^{\perp}:$
$$
T^n=
\begin{pmatrix}
P_L T^n P_L & *\\
* & *
\end{pmatrix}
$$
such that the left upper corner acts on an \emph{infinite-dimensional} space and vanishes \emph{uniformly} !

Of course, as far as we consider \emph{arbitrary} bounded operators, our approach is necessarily more delicate and involved than the ones in
  e.g. \cite{Arveson} and \cite{Hamdan}.
  A particular novelty is that in our studies of orbits we rely on the numerical ranges methodology. The  condition of orthogonality of  elements from an orbit of a bounded operator $T$ can be recasted in terms of the joint numerical range of the tuple $\mathcal T=(T, ..., T^n).$
 On the other hand, as we prove below, the joint numerical range $W(\mathcal T)$ of $\mathcal T$ contains the interior of the essential joint numerical range $W_e (\mathcal T)$ of $\mathcal T.$
 This and similar facts allow us  to construct the desired elements  from the (essential) approximate eigenvalues using inductive arguments. Other instances of these inductive arguments can be found e.g. in \cite[Chapter 5]{Muller}.
 There is certain similarity between the methods employed in this paper and a famous S. Brown's technique used for the study of invariant subspaces
 of bounded operators, see e.g. \cite{BercoviciSurv} and \cite[Chapters 3 and 4]{Chalendar}.
 To give a flavor of our results on numerical ranges, we formulate the following statement, proved
in Section \ref{numersection} (see Corollary  \ref{essencenumrange} and Theorem \ref{lambdawe}  below). It is a heart matter for subsequent considerations.
\begin{theorem}\label{main3}
Let $\mathcal T=(T_1,\dots,T_n), n \in \mathbb N,$  be an $n$-tuple of bounded linear operators on $H.$ Then
$W (\mathcal T)$ contains the interior of $W_e(\mathcal T).$
If $\mathcal T=(T,\dots,T^n)$ for some bounded linear operator $T$ on $H,$ then
the interior of $W_e (T,\dots,T^n)$ contains any tuple $(\lambda, \dots, \lambda^n)$ with $\lambda$
from the interior of the polynomial convex hull of $\sigma (T).$
\end{theorem}
Theorem \ref{main3} can be considered as a partial generalization of  \cite[Theorem 2.2]{Wrobel} dealing with
numerical ranges of operators on Banach spaces. Note that while the result from in \cite{Wrobel} allows one to
find parts of the spectrum of $\mathcal T$ in the closure $\overline{W(\mathcal T)}$ of $W(\mathcal T),$ we can replace $\overline{W(\mathcal T)}$
by a smaller and more transparent set $W(\mathcal T),$
and this has direct implications for orbits orthogonality.
%Theorem \ref{main3} extends also a result from
%\cite{Anderson} to the setting of operator tuples.

We stress that while our results mentioned above can formally be considered
as generalizations of a previous work, they are of a different nature
since we are dealing with subspaces rather than elements of a Hilbert space,
and the generality of our setting necessitates the use of new ideas.

\section{Notation}
It will be convenient to fix some of the notations in a separate section. In particular, we let
$H$ be a Hilbert space with the inner product $\langle\cdot ,\cdot \rangle,$ and $B(H)$ the space of all bounded linear operators on $H$.
For a bounded linear operator $T$ on $H$ we denote by $\sigma (T)$ its spectrum, by $r(T)$ its spectral radius and by $N(T)$ its
kernel.

For a closed set $K \subset \mathbb C^n $ we denote by $\partial K$ the topological boundary of $K,$ by $\overline{K}$ the closure of $K,$ by ${\rm Int} \, K$ the interior of $K,$
by ${\rm conv}\, K$ the convex hull of $K,$ and by $\widehat K$
the polynomial convex hull of $K$. If $K \subset \mathbb C$ then $\widehat K$ is the union of $K$ with all bounded components of the complement $\CC\setminus K$. (In that case taking $\widehat K$ can be viewed as filling ``holes'' that might exist in $K.$)

Finally, we let $\mathbb T$ stand for the unit circle $\{\lambda \in \mathbb C: |\lambda|=1\}.$

\section{Preliminaries}\label{prelim}
We start with recalling certain basic notions and facts from the spectral theory of operator tuples on Hilbert spaces.
They can be found e.g. in \cite[Chapters II.9,10 and III.18,19]{Muller}. See also \cite{Fillmore} and \cite{Dash}.
In the following we consider an $n$-tuple $\mathcal T=(T_1,\dots,T_n)\in B(H)^n$, $n \in \mathbb N.$ Note that we do not in general
assume that the operators $T_j$ commute. For $x,y\in H$ we write shortly $\langle \mathcal Tx,y\rangle= (\langle T_1 x,y\rangle,\dots,\langle T_n x,y\rangle)\in\CC^n$ and ${\mathcal T}x=(T_1x,\dots,T_nx)\in H^n$.
Similarly for $\la=(\la_1,\dots,\la_n)\in\CC^n$ we write $\mathcal T-\la=(T_1-\la_1,\dots,T-\la_n)$ and $\|\la\|_{\infty}=\max\{|\la_1|,\dots,|\la_n|\}$.

If $T_i, 1 \le i \le n,$ mutually commute then for  $\mathcal T=(T_1,\dots,T_n)\in B(H)^n$ we denote by  $\sigma({\mathcal T})$ its  joint (Harte) spectrum.
Recall that $\sigma(\mathcal T)$ can be defined as the complement to the set of those $\lambda=(\la_1,\dots,\la_n) \in\mathbb C^n$ for which
$$
\sum_{i=1}^n L_i(T_i-\lambda_i)=\sum_{i=1}^{n}(T_i-\lambda_i)R_i=I
$$
for some $L_i, R_i, 1 \le i \le n,$ from the algebra $B(H).$
If $n=1$ then the joint spectrum as above reduces to the usual spectrum of a single operator.

We will also be using a finer and  somewhat more transparent  notion of the approximate point spectrum $\sigma_\pi(\mathcal T)$ of $\mathcal T$ defined  by
$$
\sigma_\pi(\mathcal T):=\Bigl\{ \lambda=(\lambda_1, ..., \lambda_n) \in \mathbb C^n: \inf_{x \in H ,\|x\|=1} \sum_{j=1}^n\|(T_j -\lambda_j)x\|=0\Bigr\}.
$$
It is well-known that $\sigma({\mathcal T})$ and $\sigma_{\pi}({\mathcal T})$ are non-empty compact subsets of $\mathbb C^n$ and $\sigma_{\pi}({\mathcal T})\subset\si({\mathcal T})$.
There are also other joint spectra of $n$-tuples of commuting operators studied in the literature, for example the Taylor spectrum.
However, in this paper we speak only about the polynomial convex hull of the joint spectrum which coincide for all reasonable joint spectra.
 %(in particular, for all three types of the joint spectra mentioned above).

For $n \in \mathbb N$ let $\mathcal T=(T_1,\dots,T_n)\in B(H)^n$ be an $n$-tuple of commuting operators.
One can define the joint essential spectrum $\si_e(\mathcal T)$ as the (Harte) spectrum of the $n$-tuple $(T_1+{\mathcal K}(H),\dots,T_n+{\mathcal K}(H))$ in the Calkin algebra $B(H)/{\mathcal K}(H)$, where ${\mathcal K}(H)$ denotes the ideal of all compact operators on $H$.

For the purposes of this paper, the notion of  the joint essential approximate point spectrum $\sigma_{\pi e}(\mathcal T) $ will be crucial.
Recall that $\sigma_{\pi e}(\mathcal T) $ is the set of all $\la=(\la_1,\dots,\la_n)\in\CC^n$ such that
$$\inf_{x\in M,\|x\|=1}\sum_{j=1}^n\|(T_j-\la_j)x\|=0$$
for every subspace $M\subset H$ of finite codimension.
Again $\si_{\pi e}({\mathcal T})\subset\si_e({\mathcal T})$ and the polynomial convex hulls $\widehat \sigma_e({\mathcal T})$ and $\widehat \sigma_{\pi e }(\mathcal T)$  coincide (see \cite[Corollary III.19.16]{Muller}).

If $n=1$ then
$\si_{e}(T_1)=\{\la_1\in\CC: T_1-\la_1\hbox{ is not Fredholm}\}$ and
$\si_{\pi e}(T_1)=\{\la_1\in\CC: T_1-\la_1\hbox{ is not upper semi-Fredholm}\}.$
Note that the topological boundary of $\partial \sigma(T_1)$ is contained in $\sigma_{\pi}(T_1).$
Analogously,  $\partial \sigma_e(T_1) \subset \sigma_{\pi e}(T_1).$ (Such inclusions are not true anymore for $n\ge 2,$ see e.g. \cite[Section 2.5]{WrobelSt}). Moreover,
$ \partial \sigma (T_1)\setminus \sigma_{\pi e}(T_1)$ and $\sigma (T_1)\setminus \widehat \sigma_{e}(T_1)$
consist of isolated points of $\sigma (T_1)$ (in fact of eigenvalues of $T_1$ of finite multiplicity), see e.g. \cite[p. 184]{Fillmore} and \cite[Theorem III.19.18]{Muller}.
Thus, in particular,
\begin{equation}\label{inclusion}
\mathbb T \subset \sigma (T_1), \,\, r(T_1) \le 1 \qquad \Longrightarrow \qquad \mathbb T \subset \sigma_{\pi e} (T_1).
\end{equation}
If $T\in B(H)$ and $\mathcal T=(T,T^2, \dots, T^n) \in B(H)^n,$ then $\sigma (\mathcal T)=\{(\lambda, \dots, \lambda^n): \lambda \in \sigma (T)\}$ and
$\sigma_\pi (\mathcal T)=\{(\lambda, \dots, \lambda^n): \lambda \in \sigma_\pi (T)\}$. Similar relations are true for the essential spectrum $\si_e$ and essential approximate point spectrum $\si_{\pi e}$.
For the essential spectrum theory in the realm of Hilbert spaces one may also consult \cite{Fillmore} for the case of single operators and \cite{Dash} for the case of $n$-tuples.

As in the case of a single operator, for $\mathcal T=(T_1,\dots,T_n)\in B(H)^n$ it is often useful to relate $\sigma(\mathcal T)$ to a larger and easier computable set $W(\mathcal T)\subset \mathbb C^n$ called the joint numerical
range of $\mathcal T$ and defined as
$$
W(\mathcal T)=\{(\langle T_1 x, x\rangle , ..., \langle T_n x, x \rangle) : x \in H, \|x\|=1\}.
$$
Unfortunately, if  $n>1,$ then $W(\mathcal T)$ is not in general convex, see e.g \cite{Li}.

As in the spectral theory, there is also a notion of the joint essential numerical range  $W_e(\mathcal T)$  associated to $\mathcal T.$
The set $W_e(\mathcal T)$ will be of major importance in our arguments, and it can be described as the set of all $n$-tuples
$\la=(\la_1,\dots,\la_n)\in\CC^n$ such that there exist an orthonormal sequence $(x_k)\subset H$ with
$$
\lim_{k\to\infty}\langle  T_j x_k,  x_k\rangle=\la_j, \qquad j=1,\dots,n.
$$
 Clearly, $W_e(\mathcal T)\supset \sigma_e (\mathcal T).$ Note that $\la\in W_e(\mathcal T)$ if for every subspace $M\subset H$ of finite codimension and every $\de>0$ there exists a unit vector $x\in M$ such that $\|\langle \mathcal T x,x\rangle-\la \|_{\infty}<\de.$
Alternatively, $W_e(\mathcal T)$ can be defined as
$$
W_e(\mathcal T):= \bigcap \overline{W(T_1 + K_1,\dots,T_n+K_n)},
$$
where the intersection is taken over all $n$-tuples $K_1,\dots,K_n$ of compact operators on $H.$
Recall that $W_e(\mathcal T)$ is a compact and, in contrast to $W(\mathcal T),$ \emph{convex} subset of $\overline{W(\mathcal T)},$ see \cite{Li}.
Moreover, if $\mathcal T$ consists of commuting operators, then since $W_e(\mathcal T)$ is convex, $\sigma_{\pi
e}( \mathcal T)\subset W_e( \mathcal T)$  and
the convex hulls of  $\sigma_e(\mathcal T)$ and $\sigma_{\pi e}(\mathcal T)$ coincide (see the proof of Corollary \ref{convex}), one has $W_e(\mathcal T)\supset {\rm conv} \, \sigma_e (\mathcal T).$
For a comprehensive account of essential numerical ranges one may consult \cite{Li} and the references therein.

\section{Spectra and numerical ranges for tuples}\label{numersection}

The next proposition will be instrumental in approximating numerical ranges by spectra, and in relating spectra to orthogonality relations.
\begin{proposition}\label{basic}
Let $\mathcal T=(T_1,\dots,T_n)\in B(H)^n$, $k\in\NN\cup\{0\}$, $r>0$, and
$$\{\xi=(\xi_1,\dots,\xi_n):\|\xi\|_{\infty} \le r\}\subset W_e(\mathcal T).$$
Suppose $M\subset H$ is  a subspace of a finite codimension, and $x\in M$ satisfies
$$
\|x\|^2=1-2^{-k} \qquad \text{and} \qquad |\langle T_jx,x\rangle|\le {r}{2^{-k-1}}, \qquad j=1,\dots,n.
$$
Then there exists $x'\in M$ such that
$$
\|x'\|^2=1-2^{-k-1}, \qquad \|x'-x\|^2={2^{-k-1}} \quad  \text{and} \quad |\langle T_jx',x'\rangle|\le
{r}{2^{-k-2}}
$$
for all $j=1,\dots,n$.

Consequently, there exists $w\in M$ such that
$$
\|w\|=1, \qquad \|w-x\|\le 3\cdot 2^{-\frac{k}{2}-1} \qquad \text{and} \qquad \langle T_jw,w\rangle=0
$$
for all $j=1,\dots,n$.
\end{proposition}

\begin{proof}
Let $\xi=\langle\mathcal T x,x\rangle,$ so that $\|\xi\|_{\infty} \le \frac{r}{2^{k+1}}$ and $-2^{k+1}\xi\in W_e(\mathcal T)$.
Note that
$$
L:= \bigvee \{x, T_jx, T^{*}_jx: j=1,\dots,n \}
$$
is a finite-dimensional subspace of $H$.
Thus, by assumption, there exists a unit vector $u\in M\cap L^{\perp}$ such that
$$
\|\langle T_ju,u\rangle+\xi 2^{k+1}\|_{\infty}<\frac{r}{2}, \qquad j=1,\dots,n.
$$
Set $x'=x+ 2^{-\frac{k+1}{2}}u$. Then
\begin{align*}
\|x' \|^2&=\|x\|^2+\frac{1}{2^{k+1}}=1-2^{-k-1},\\
 \|x'-x\|^2&=2^{-k-1},
\end{align*}
 and
 \begin{align*}
|\langle T_jx',x'\rangle|=
\bigl|\langle T_jx, x\rangle +\frac{1}{2^{k+1}}\langle T_ju,u\rangle\bigr|=
\bigl|\xi_j+\frac{1}{2^{k+1}}\langle T_ju,u\rangle\bigr|
\le
\frac{r}{2^{k+2}}
\end{align*}
for all $1 \le j \le n.$ This finishes the proof of the first part of the proposition.

To prove its second part, we construct $w$ as the limit of an appropriate sequence $(x_m), m \ge k$.
To construct the sequence,
set $x_k=x\in M$. We have
$$\|x_k\|^2=1-2^{-k} \qquad \text{and} \qquad \|\langle \mathcal Tx_k,x_k\rangle\|_{\infty} \le \frac{r}{2^{k+1}}.$$
If we put $x_{k+1}=x',$ then by the first part of the proposition,
$$\|x_{k+1}\|^2=1-\frac{1}{2^{k+1}}, \qquad \|x_{k+1}-x_k\|^2=\frac{1}{2^{k+1}} \quad \text{and}\quad
\|\langle\mathcal T x_{k+1},x_{k+1} \rangle\|_{\infty}\le\frac{r}{2^{k+2}}.
$$
Thus, repeating the procedure above, we construct inductively vectors $x_{m} \in M, m \ge k,$ such that
$$\|x_m\|^2=1-\frac{1}{2^m}, \qquad \|x_{m+1}-x_m\|^2=\frac{1}{2^{m+1}} \quad \text{and}\quad
\|\langle\mathcal T x_m,x_m\rangle\|_{\infty}\le\frac{r}{2^{m+1}}.
$$
Clearly the sequence $(x_m)$ is Cauchy. Let $w$ be its limit. By construction,
$$
w\in M, \qquad \|w\|=1, \qquad \langle T_jw,w\rangle=0,
$$
for all $1 \le j \le n,$ and
$$
\|w-x\|\le \sum_{m=k}^\infty\|x_{m+1}-x_m\| =2^{-k/2}\frac{1}{\sqrt{2}-1}< 3\cdot 2^{-\frac{k}{2}-1}.
$$

\end{proof}

Proposition \ref{basic} implies in particular that points from the interior of $W_e (T_1, \dots, T_n)$
belong to $W(T_1, \dots, T_n)$ and, moreover, can be attained on any subspace of $H$ of finite codimension.

\begin{corollary}\label{essencenumrange}
Let $\mathcal T=(T_1,\dots,T_n)\in B(H)^n.$  Then
$$\Int \, W_e(\mathcal T) \subset W(\mathcal T).$$
Moreover, if $\la=(\la_1, \dots, \la_n)\in\Int \, W_e(\mathcal T)$
then for every subspace $M\subset H$ of a finite codimension there exists $x\in M$ such that $\|x\|=1$ and
$$
 \bigl(\langle T_1 x, x\rangle,\dots,\langle T_nx,x\rangle\bigr)=\la.
$$
\end{corollary}

\begin{proof}
Without loss of generality we may assume that $\la=(0,\dots,0)$ (by considering the $n$-tuple $(T_1-\la_1,\dots,T_n-\la_n)$ instead of $\mathcal T$).

Let $k=0$ and $x=0$. Then Proposition \ref{basic} yields the statement.
\end{proof}
Another interesting consequence of Proposition \ref{basic}
allows one to find a \emph{joint} diagonal compression for $T_1, \dots, T_n$
to an infinite-dimensional subspace of $H.$ This can be considered as non-commutative generalization of the technique (and statements)
employed in  e.g. \cite{Anderson}, \cite{Bourin} and other papers dealing with compressions with help
of essential numerical ranges. The statement below was proved in \cite[p.440]{Anderson} for $n=1.$
(For $T \in B(H)$ the problem of characterizing $\lambda \in \mathbb C$ such that $PTP=\lambda P$ for an infinite rank projection $P$ was posed in \cite[p.190]{Fillmore}.)
\begin{corollary}
Let $\mathcal T=(T_1,\dots,T_n)\in B(H)^n$ and $\la=(\la_1,\dots,\la_n) \in\Int \, W_e(\mathcal T)$. Then there exists an infinite-dimensional subspace $L$ of $H$ such that
$$
P_LT_jP_L=\la_jP_L, \qquad j=1,\dots,n,
$$
where $P_L$ is the orthogonal projection on $L.$
\end{corollary}

\begin{proof}
 Using Corollary \ref{essencenumrange}, find a unit vector $x_1\in H$ such that $\langle\mathcal T x_1,x_1\rangle=\la$.
Construct inductively a sequence $(x_k)\subset H$ of unit vectors such that
$$
x_{k+1}\perp\{x_{m}, T_jx_m,T_j^{*}x_m: 1\le m\le k, 1\le j\le n\}
$$
and
$$
\langle\mathcal T x_k,x_k\rangle=\la
$$
for all $k\in\NN$ using the fact that
  $$\bigvee \{x_{m}, T_jx_m, T_j^{*}x_m: 1\le m\le k, 1\le j\le n\}$$
is a subspace of finite dimension.
Let $L=\bigvee_{k=1}^\infty x_k$. Clearly $L$ is an infinite-dimensional subspace with an orthonormal basis $(x_k)$.
Let $y\in L.$ Then, in view of our construction of $(x_k),$ it is easy to see that
$$
\langle \mathcal T y,y\rangle =\la \|y\|^2.
$$
Hence $P_L T_j P_L=\la_j P_L$ for all $1 \le j \le n.$
\end{proof}

The next result allows one also to describe a ``large'' subset of $W(\mathcal T)$
in purely spectral terms.

\begin{corollary}\label{convex}
Let $\mathcal T=(T_1,\dots,T_n)\in B(H)^n$ be a $n$-tuple of commuting
operators. Then $$\Int\,\conv\, \sigma_e(\mathcal T)\subset W(\mathcal T).$$
\end{corollary}
\begin{proof}
Since the polynomial convex hulls of $\si_e(\mathcal T)$ and $\si_{\pi e}(\mathcal T)$ coincide (see \cite[Proposition III.19.15]{Muller}),
we have $\conv\si_e(\mathcal T)=\conv\si_{\pi e}(\mathcal T)$.
Then, taking into account that $\si_{\pi e}(\mathcal T)\subset W_e(\mathcal T)$, we infer that
$$
\Int\conv\si_e(\mathcal T)=
\Int\conv\si_{\pi e}(\mathcal T)
\subset\Int W_e(\mathcal T)\subset W(\mathcal T).
$$
\end{proof}

To clarify further the interplay between joint spectra and numerical ranges, we will need the next statement on interpolation of points
 from the polynomial hull of a compact $K \subset \mathbb C$  by the convex hull of ``powers'' of $K.$

\begin{proposition}\label{4.5}
Let $K\subset\CC$ be a compact set and $0\in\Int\widehat K$. Let $n\in\NN$. Then
there exists $r>0$ such that the following is true.
For every $\xi=(\xi_1,\dots,\xi_n)\in\CC^n$ with $\|\xi\|_{\infty}\le r,$ there are $m\in\NN$, $\la_1,\dots,\la_m\in K$ and $c_1,\dots,c_m\ge 0$ satisfying  $\sum_{j=1}^mc_j=1$ and
$$
\sum_{j=1}^mc_j\la_j^k=\xi_k, \qquad k=1,\dots,n.
$$
\end{proposition}

\begin{proof}
Note that if  $q$ is a nonzero polynomial such that $q(0)=0,$ then
$$
\sup\{\Re q(\lambda):\lambda\in K\} >0
$$
by the maximum principle for the harmonic function $\Re q$.

By a compactness argument, there exists a $\de>0$ satisfying
$$
 \sup\{\Re p(\lambda):\lambda\in K\} >\de
$$
for every polynomial  $p(\lambda)=\sum_{j=1}^n\alpha_j\lambda^j$ with $\sum_{j=1}^n|\alpha_j|=1$.
Hence there exists $r>0$ such that
$$
\Re\sum_{j=1}^n \al_j\xi_j<\sup\{\Re p(\lambda):\lambda\in K\}
$$
for every $\xi=(\xi_1,\dots,\xi_n)\in\CC^n$ with $\|\xi\|\le r$ and every polynomial $p$ as above.

 Let $A=\hbox{conv} \,\{(\lambda, \dots, \lambda^n): \lambda\in K\}$. Then $A$ is compact as the convex hull of a compact set in $\CC^n.$  .
Let now $\xi\in\CC^n$ satisfy $\|\xi\|_{\infty}\le r$.
If $\xi\notin A$, then by the Hahn-Banach theorem there exists a functional $\varphi$ in the dual space of $(\CC^n), \|\cdot\|_\infty)$
with $\|\varphi\|=1$ such that
$$
\Re\varphi(\xi)>\sup\{\Re\varphi((\lambda, \dots, \lambda^n)):\lambda\in K\}.
$$
So there exist  $\{\al_1,\dots,\al_n \} \subset \mathbb C, \sum_{j=1}^n|\alpha_j|=1,$ such that
$$
\Re\sum_{j=1}^n\al_j\xi_j>\sup\Bigl\{\Re\sum_{j=1}^n\al_j\lambda^j:\lambda\in K\Bigr\},
$$
which is a contradiction with our choice of $r.$ Thus $\xi\in A$, and the proof is complete.
\end{proof}

Now we are ready to prove the statement which will  also be basic for constructions of orbits
with orthogonality properties, and will complement Proposition \ref{basic}.
\begin{theorem}\label{lambdawe}
Let $T\in B(H)$ and let $\la\in\Int\, \widehat\sigma(T)$.  Then
$$
(\la,\la^2,\dots,\la^n)\in\Int \, W_e(T,T^2,\dots,T^n).
$$
for all $n\in\NN.$
\end{theorem}
\begin{proof}

Assume first that $\la=0$.
Then
$$0\in \Int\, \widehat\sigma(T)=\Int \, \widehat\sigma_e(T)=\Int\, \widehat \sigma_{\pi e}(T)$$
(see Section \ref{prelim} or \cite[Corollary III. 19.16 and Theorem III.19.18]{Muller}).

We apply Proposition \ref{4.5} to the compact set $K=\sigma_{\pi e}(T)$.
Let $r>0$ be given by Proposition \ref{4.5}.
Let $\xi_1,\dots,\xi_n\in\CC$, $\max_{1 \le j \le n} |\xi_j|\le r$ and $\de>0$. Let $M\subset H$ be a subspace of a finite codimension.
We show that there exists a unit vector $x\in M$ such that
$$|\langle T^jx,x\rangle-\xi_j|<\de, \qquad j=1,\dots,n.$$

By Proposition \ref{4.5}, there exist $m\in\NN$, $\la_1,\dots,\la_m\in \sigma_{\pi e}(T)$ and numbers $c_i\ge 0$ with $\sum_{i=1}^mc_i=1$ such that
$$
\sum_{i=1}^mc_i\la_i^j= \xi_j, \qquad j=1,\dots,n.
$$
Let $$0<\de'<\frac{\de}{n\cdot \max\{1, \|T\|^n\}}.$$ Since $\la_i\in\sigma_{\pi e}(T), 1 \le i \le m,$ we can find inductively
unit vectors $x_i\in M$ such that
\begin{align*}
x_i\perp x_k,& \qquad i\ne k, \, 1\le i,k\le m,\\
x_i\perp T^j x_k,& \qquad  j=1,\dots,n, \, i\ne k, \, 1\le i,k\le m,\\
\|Tx_i-\la_ix_i\|\le \de',& \qquad 1 \le i \le m.
\end{align*}
Note that for every $n \in \mathbb N$ and all $j$ such that $1\le j\le n$ we then have
$$\|T^jx_i-\la_i^jx_i\|\le \de' n \max\{1,\|T\|^n\}<\de.$$
Set $x=\sum_{i=1}^m c_i^{1/2}x_i$. Then $x \in M$ and $\|x\|^2=\sum_{i=1}^mc_i=1$. If $1\le j\le n$ then
\begin{align*}
|\langle T^jx,x\rangle-\xi_j|=&
\Bigl|\sum_{i=1}^m c_i\langle T^jx_i,x_i\rangle-\xi_j\Bigr|\\
\le&
\sum_{i=1}^m c_i \|T^jx_i-\la_i^jx_i\| +
\Bigl|\sum_{i=1}^m c_i\langle \la_i^j x_i,x_i\rangle-\xi_j\Bigr|
%\le&
%\de+\Bigl|\sum_{i=1}^{m} c_i\la_i^j-\xi_j\Bigr|
\le \de.
\end{align*}
Since $\de>0$ and $M\subset H$, $\codim M<\infty$ were arbitrary, we have
$(0,\dots,0)\in\Int \, W_e(T,T^2,\dots,T^n)$.

Let now $\lambda\in\Int\widehat\sigma(T)$ be arbitrary, and
set $S=T-\lambda$.
Thus, $0 \in \Int\, \widehat\sigma(S),$ and then, as we have proved above,
$$(0,\dots,0)\in\Int\, W_e(S, S^2,\dots,S^n).$$
Observe that for each $n,$
$$(T, T^2, \dots, T^n)=\bigl(S+\lambda, S^2 +2\lambda S +\lambda^2, \dots, \sum_{j=0}^n {n\choose
j}S^j\lambda^{n-j}\bigr).$$

Let the mapping $G:(\CC^n, \|\cdot\|_{\infty}) \to (\CC^n \|\cdot\|_{\infty})$ be defined by
$$
G(z_1,\dots,z_n)= (z_1+\lambda, z_2+2\lambda z_1+\lambda^2, \dots ,
\sum_{j=1}^n {n\choose j} z_j\lambda^{n-j} + \lambda^n ).
$$
Note that the mapping $$(z_1,\dots,z_n)\mapsto
G(z_1,\dots,z_n)-(\lambda,\lambda^2,\dots,\lambda^n)$$
is linear and invertible (since it is determined by an upper triangular
matrix with non-zero diagonal). So $G$ maps any neighbourhood of
$(0,\dots,0)$ onto a neighbourhood of
$(\lambda,\lambda^2,\dots,\lambda^n)$.
Using the definition of $W_e(T,T^2,\dots,T^n),$ it is easy to see that
$$W_e(T,T^2,\dots,T^n)=\bigl \{G(z_1,\dots, z_n): (z_1,\dots,z_n)\in
W_e(S,S^2,\dots,S^n) \bigr \}.$$
(Note that a similar relation holds also for $W(T, \dots, T^n)$).
Hence, we infer that  $$(\la,\la^2,\dots,\la^n)\in\Int \, W_e(T, \dots, T^n),$$ and the theorem follows.
\end{proof}
%Since $\Int \, \hat \sigma_e(T)=\Int \,\hat \sigma (T)$ (see Section \ref{prelim})
The following corollary is an immediate consequence
of Theorem  \ref{lambdawe} and Corollary \ref{essencenumrange}.
\begin{corollary}\label{lambdaw}
Let $T\in B(H)$ and let $\la\in\Int\, \widehat\sigma(T).$ Then
$$
(\la,\la^2,\dots,\la^n)\in W(T,T^2,\dots,T^n)
$$
for all $n\in\NN.$
Moreover, for each $n \in \mathbb N$
there exists an infinite-dimensional subspace $L\subset H$ such that
$$
P_LT^jP_L=\la^jP_L, \qquad j=1,\dots,n,
$$
where $P_L$ is the orthogonal projection on $L.$
\end{corollary}

\section{Circles in the spectrum and orthogonality}\label{circles}

In this section we  characterize operators having the unit circle in their spectra
by means of orthogonality (and ``almost orthogonality'') properties
of their orbits. We start with the proof of Theorem \ref{main1} stated in the introduction.
Our arguments will be based on numerical ranges considerations from the previous section.

\noindent
{\it Proof of Theorem \ref{main1}.} \,\,
The implication (iii)$\Rightarrow$(ii) is obvious.
\smallskip

(ii)$\Rightarrow$(i):
Let $\lambda \in \mathbb C$ be such that $|\la|=1$.
For $k\in\NN$ let $\e=k^{-3}$ and $n=(k+1)^2$. Let $x_k$ be a vector $x$ satisfying (ii)
for such $\epsilon$ and $n$.

Set $$y_{k,0}=x_k+\la^{-1}Tx_k+\cdots+ \la^{-k+1}T^{k-1}x_k.$$
Then
$$
\|(T-\la)y_{k,0}\|=
\bigl\|-\la x_k+\la^{-k+1}T^kx_k\bigr\|\le 4
$$
and
\begin{align*}
\|y_{k,0}\|^2&=
\sum_{j,m=0}^{k-1}\langle \la^{-j}T^jx_k,\la^{-m}T^mx_k\rangle\\
&\ge
\sum_{j=0}^{k-1}\|T^jx_k\|^2-\sum_{0\le j,m\le k-1\atop j\ne m}\bigl|\langle \la^{-j}T^jx_k,\la^{-m}T^mx_k\rangle\bigr|
\\
&\ge
\frac{k}{4}-k^2\e
=\frac{k}{4}-\frac{1}{k}.
\end{align*}
Let $u_{k,0}=\frac{y_{k,0}}{\|y_{k,0}\|}$. Then $\|u_{k,0}\|=1$ and $\lim_{k\to\infty}\|(T-\la)u_{k,0}\|=0$. Hence $\la\in\si_{\pi}(T)$.

Suppose on the contrary that $\la\notin\si_{\pi e}(T)$. Then, by \cite[Theorem III.16.8]{Muller}, the operator $T-\la$ is upper semi-Fredholm,
that is  $\dim N(T-\la)<\infty$ and $T \upharpoonright_{N(T-\la)^\perp}$ is bounded below.
 Let $P$ be the orthogonal projection onto $N(T-\la)$.
Let $x_k, y_{k,0}$ and $u_{k,0}$ be as above.
Then $(T-\la)u_{k,0}\to 0, k \to \infty$.
Since $(T-\la)P=0$, we also have $$(T-\la)(I-P)u_{k,0}\to 0, \quad k \to \infty,$$ and so
$(I-P)u_{k,0}\to 0, k \to \infty$. Since the unit ball in $N(T-\la)$ is compact, we can assume (by passing to a subsequence if necessary) that $Pu_{k,0}\to v_0, k \to \infty,$ and $v_0 \in N(T-\la)$. Hence
$$u_{k,0}\to v_0, \quad k \to \infty, \qquad \text{and} \qquad \|v_0\|=1.$$

For $j=1,\dots,k$ set
$y_{k,j}=T^{kj}y_{k,0}$ and note that as above $\|y_{k,j}\|\ge \frac{k}{4}-\frac{1}{k}.$
Let $$u_{k,j}=\frac{y_{k,j}}{\|y_{k,j}\|}, \qquad 1 \le j \le k.$$
In the same way as for $u_{k,0}$ one can show that $$\lim_{k\to\infty}(T-\la)u_{k,j}=0$$ for all $j \in \mathbb N.$
Moreover one can assume that $$\lim_{k\to\infty}u_{k,j}=v_j\in N(T-\la),$$ where $\|v_j\|=1$
for all $j$.
If  $j\ne m$ then
$$
|\langle u_{k,j},u_{k,m}\rangle|\le
\sum_{s,s'=0}^{k-1}|\langle T^{s+kj}x_k,T^{s'+km}x_k\rangle|\le k^2\e=k^{-1}.
$$
So $\langle v_j,v_m\rangle=0$ for all $j,m\in\NN, j\ne m$. Hence
$\dim N(T-\lambda)=\infty$, a contradiction.
\smallskip

(i)$\Rightarrow$(iii):
Let $\e>0$ and $n\in\NN$ be fixed. Note that $\|T\|\ge 1$.

Using the assumption that $\TT\subset\sigma_{\pi e}(T),$ find inductively unit vectors $u_0,u_1,\dots,u_{n-1}$ such that
$$
\langle T^ju_k,T^{j'}u_{k'}\rangle=0,\qquad 0\le k,k'\le n-1, k\ne k', 0\le j,j'\le n-1,
$$
and
$$
\|Tu_k-e^{2\pi i k/n}u_k\|<\frac{\e}{4n^{3/2}\|T\|^{2n}}, \qquad 0\le k\le n-1.
$$
For $1\le j\le n-1$ and $0\le k\le n-1$ we have
\begin{align*}
&\|T^j u_k-e^{2\pi i kj/n}u_k\|\\
&\le
\bigl\|T^{j-1}+T^{j-2}e^{2\pi ik/n}+\cdots+e^{2\pi ik(j-1)
/n}\bigr\|\cdot\|Tu_k-e^{2\pi i k/n}u_k\|\\
&\le
\frac{\e}{4n^{1/2}\|T\|^n}.
\end{align*}
Set $$v:=\frac{1}{\sqrt{n}}\sum_{k=0}^{n-1}u_k.$$ Then $\|v\|=1$.
If $0\le j\le n$ then
$$
\Bigl\|T^jv-\frac{1}{\sqrt{n}}\sum_{k=0}^{n-1}e^{2\pi ikj/n}u_k\Bigr\|\le
\frac{1}{\sqrt{n}}\sum_{k=0}^{n-1}\bigl\|T^j u_k-e^{2\pi i kj/n}u_k\bigr\|
\le
\frac{\e}{4\|T\|^n}\le\e/4.
$$
So for $0\le j,m\le n-1, j\ne m$ it follows that
\begin{align*}
&|\langle T^jv,T^mv\rangle|\le
\Bigl\|T^jv-\frac{1}{\sqrt{n}}\sum_{k=0}^{n-1}e^{2\pi ijk/n}u_k\Bigr\|\cdot\|T^mv\|\\
+&\Bigl\|\frac{1}{\sqrt{n}}\sum_{k=0}^{n-1} e^{2\pi i jk/n} u_k\Bigr\|\cdot
\Bigl\|T^mv-\frac{1}{\sqrt{n}}\sum_{k=0}^{n-1}e^{2\pi imk/n}u_k\Bigr\|+
\frac{1}{n}\Bigl|\sum_{k=0}^{n-1} e^{2\pi i (j-m)k/n}\Bigr|\\
\le&
\|T^m\|\cdot\frac{\e}{4\|T\|^n}+\frac{\e}{4\|T\|^n}
\le \e/2.
\end{align*}
Similarly,
$$
\|T^nv-v\|=\Bigl\|T^nv-\frac{1}{\sqrt{n}}\sum_{k=0}^{n-1}e^{2\pi i k}u_k\Bigr\|<
\frac{\e}{4\|T\|^n}\le\e/4.
$$
Finally,
$$\Bigl\|\frac{1}{\sqrt{n}}\sum_{k=0}^{n-1}e^{2\pi ijk/n}u_k\Bigr\|=1, \qquad  j=0,1,\dots,n-1,$$
 and so for every $j$ such that $1 \le j \le n-1,$
$$
1-\e/4\le \|T^jv\|\le 1+\e/4,.
$$

Choose now $c\in\NN$ such that $$32\cdot 2^{-c/2}\|T\|^n<\e,$$
 and let $w=(1-2^{-c})^{1/2}v$. Then
$$\|w\|^2=1-2^{-c}\qquad \text{and}\quad |\langle T^jw,w\rangle|\le |\langle T^jv,v\rangle|<\e/2$$
for each $1 \le j \le n-1.$
Since $\TT\subset\sigma_{\pi e}(T)$, we have
$$0 \in  \Int\, \widehat \sigma_e(T).$$
Hence, by Theorem \ref{lambdawe}, it follows that
$$(0, \dots, 0) \in {\rm Int}\, W_e(T, \dots, T^{n-1}).$$
Then Proposition \ref{basic} implies that there exists a unit vector $x\in H$ such that
$$ \langle T^jx,x\rangle=0, \qquad j=1,\dots,n-1, \quad \text{and} \quad \|x-w\|= 3\cdot 2^{-c/2}.
$$
So
$$
\|x-v\|\le\|x-w\|+\|w-v\|\le
3\cdot 2^{-c/2}+(1-\sqrt{1-2^{-c}})\le
4\cdot 2^{-c/2}<\frac{\e}{8\|T\|^n}.
$$
If $0\le j\le n-1$ then
$$
\|T^jx\|\le \|T^jv\|+\|T^jx-T^jv\|\le
1+\frac{\e}{4}+\frac{\e}{8}<1+\e,
$$
and similarly,
$$
\|T^jx\|\ge \|T^jv\|-\|T^j x-T^jv\|\ge 1-\frac{\e}{4}-\frac{\e}{8}>1-\e.
$$
Moreover,
$$
\|T^nx-x\|\le
\|T^nx-T^nv\|+\|T^nv-v\|+\|v-x\|\le
(\|T\|^n+1)\|x-v\|+\e/4<\e.
$$
Finally, for $1 \le j,m\le n-1, j\ne m,$
\begin{align*}
|\langle T^jx,T^mx\rangle|&\le
\|T^jx-T^jv\|\cdot\|T^mx\|+ \|T^jv\|\cdot\|T^mx-T^mv\|+|\langle T^jv,T^mv\rangle|
\\
&\le
4\|T\|^n\cdot\|x-v\|+\e/2
<\e.
\end{align*}
Hence $x$ satisfies all conditions of (iii).
\hfill $\Box$
%\end{proof}

\begin{remark}
Note that the argument given in the beginning of the proof of (i)$\Rightarrow$(iii) easily yields that the $n$-tuple $(0, \dots, 0)$ belongs to
$W_e(T, \dots, T^n).$ The reason for invoking Theorem \ref{lambdawe} is that one needs
to show that $(0, \dots, 0)$ belongs to the interior of $W_e(T, \dots, T^n)$ in order to be able
to apply Proposition \ref{basic}.
\end{remark}

The following result shows that under mild assumptions one can replace essential spectrum by spectrum in Theorem \ref{main1}.

\begin{theorem}\label{circlepb}
Let $T\in B(H)$ satisfy $r(T)\le 1$. The following statements are equivalent.
\begin{itemize}
\item [(i)] $\TT\subset\sigma(T).$

\item[(ii)] for every $\e>0$ and every $n\in\NN$ there exists a unit vector $x\in H$ such that
$$
|\langle T^mx,T^jx\rangle|<\e, \qquad 0\le m,j\le n-1, m\ne j,
$$
and
$$
\|T^nx-x\|<\e.
$$

\item [(iii)] for every $\e>0$ and every $n\in\NN$ there exists a unit vector $x\in H$ such that
\begin{align*}
x\perp T^{j}x,& \qquad 1 \le j \le n-1,\\
|\langle T^mx,T^jx\rangle|< \e,& \qquad 1\le m,j\le n-1, m\ne j,\\
1-\e<\|T^jx\|< 1+\e,& \qquad 1\le j\le n-1,
\end{align*}
and
$$
\|T^nx-x\|<\e.
$$
\end{itemize}
\end{theorem}

\begin{proof}
The implication (iii)$\Rightarrow$(ii) is obvious.
\smallskip

(i)$\Rightarrow$(iii):
Since $\TT\subset\si(T)$ and $r(T)\le 1$,  we infer by \eqref{inclusion} that $\TT\subset\si_{\pi e}(T)$. So (iii) follows from Theorem \ref{main1}.

\smallskip

(ii)$\Rightarrow$(i):
Let $\lambda \in\mathbb C$ be such that $|\la|=1.$ For $k\in\NN$ fix any $n\ge k$ such that $|\la^n-1|<k^{-1},$ and  let $\e=n^{-3}$.
Let $x_k$ be the vector $x$ satisfying (ii) for $n$ and $\epsilon$ as above.

Set
$$y_k=x_k+\la^{-1}Tx_k+\cdots+ \la^{-n+1}T^{n-1}x_k.$$
Then
\begin{align*}
\|(T-\la)y_k\|=&\|-\lambda x_k+ \lambda^{-n+1} T^nx_k\|\\
=&\bigl\|-x_k+\lambda^{-n}T^n x_k \bigr \|\\
\le& \|-x_k+\lambda^{-n}x_k\| + \bigl\|-\lambda^{-n}x_k+\lambda^{-n} T^nx_k \bigr \|
\le
2k^{-1},
\end{align*}
and
\begin{align*}
\|y_k\|^2&=
\sum_{j,j'=0}^{n-1}\langle \la^{-j}T^jx_k, \la^{-j'}T^{j'}x_k\rangle\\
&=
\sum_{j=0}^{n-1}\|T^jx_k\|^2+ \sum_{0\le j,j'\le n-1\atop j\ne j'}\langle \la^{-j}T^jx_k, \la^{-j'}T^{j'}x_k\rangle\\
&\ge\|x_k\|^2-n^2\e
\ge 1-1/k.
\end{align*}
Hence $\la\in\sigma_\pi(T)\subset\si(T)$.
\end{proof}

Now we turn to the case of unitary $T.$ The next corollary of Theorem \ref{circlepb} is a strengthening of Arveson's theorem
from \cite{Arveson}
formulated in the introduction.
\begin{theorem}\label{circleunitary}
Let $T$ be a unitary operator on $H$. The following statements are equivalent.
\begin{itemize}
\item [(i)] $\sigma(T)=\TT.$

\item [(ii)] for every $\e>0$ and every $n\in\NN$  there exists a unit vector $x\in H$ such that
\begin{align*}
x\perp T^{j}x,& \qquad 1 \le j \le n-1,\\
|\langle T^mx,T^jx\rangle|<\e,& \qquad 1\le m,j\le n-1, i\ne j,
\end{align*}
and
$$
\|T^nx-x\|<\e.
$$
\item [(iii)] For every $n\in\NN$ there exists a unit vector $x\in H$ such that the vectors $x, Tx, T^2,\dots, T^{n}x$ are mutually orthogonal.
\item[(iv)] For every $n\in\NN$,
$$
(0,\dots,0)\in \overline{W(T,T^2,\dots,T^{n})}.
$$
\end{itemize}
\end{theorem}

\begin{proof}
The equivalence (i)$\Leftrightarrow$(ii) was proved in the previous theorem.
\smallskip

The implication (ii)$\Rightarrow$(iii) follows from the fact that $T$ is unitary and (iii)$\Rightarrow$(iv) is obvious.
\smallskip

(iv)$\Rightarrow$(i): See Theorem \ref{main1}, implication (ii)$\Rightarrow$(i).
\end{proof}

We finish this section with another  operator  version of the Rokhlin Lemma.
Note that we do not assume that $T$ below is unitary.

\begin{theorem}\label{connes}
Let $T\in B(H)$. The following statements are equivalent.

(i) $\TT\subset\si_{\pi e}(T)$;

(ii) for all $\e>0$, $n> \max (4\|T\|^2\e^{-2},1)$ and $u\in H$, $\|u\|=1$ there exist orthonormal vectors
$w_0,\dots,w_{n-1}\in H$ such that
$$
\|Tw_j-w_{j+1}\|<\e \qquad 0\le j\le n-2, \qquad
\|Tw_{n-1}-w_0\|<\e,
$$
and
$$
\frac{1}{\sqrt n}\sum_{j=0}^{n-1}w_j=u.
$$
\end{theorem}

\begin{proof}
(ii)$\Rightarrow$(i): Let $\lambda \in \mathbb C$ be such that $|\la|=1$, and let $u\in H$, $\|u\|=1$ be arbitrary.
 For $k\in\NN$ choose any $n> \max (4\|T\|^2\e^{-2}, k^2).$
Suppose that  $w_0,\dots,w_{n-1}$ satisfy (ii) with $\e=k^{-1},$  and
set $$y_{k,0}=w_0+\la^{-1}w_1+\cdots+\la^{-k+1}w_{k-1}.$$
 Then $\|y_{k,0}\|=\sqrt{k}$ and
\begin{align*}
&\|(\la-T)y_{k,0}\|\\
=&
\bigl\|\la w_0+  (w_1-Tw_0)+ \cdots + \la^{-k+2}(w_{k-1}-Tw_{k-2})+\la^{-k+1}Tw_{k-1}\bigr\|\\
\le& \|w_0\|+(k-1)\e+\|Tw_{k-1}\|\\
\le&
1+(k-1)\e+\|T w_{k-1}- w_k\|+\|w_k\|\\
\le& 2+ k\e
=3,
\end{align*}
so that
$$
\lim_{k\to\infty}\Bigl\|(T-\la)\frac{y_{k,0}}{\|y_{k,0}\|}\Bigr\|=0.
$$

Consider then vectors $y_{k,1},\dots,y_{k,k}, k \in \mathbb N,$ where $$y_{k,m}=w_{mk}+\la^{-1}w_{mk+1}+\cdots+\la^{-k+1}w_{mk+k-1}, \qquad 1 \le m \le k.$$
Analogously to the above one can show that
$$\lim_{k\to\infty}\Bigl\|(T-\la)\frac{y_{k,m}}{\|y_{k,m}\|}\Bigr\|=0,$$
for all $m$ such that $1 \le m \le k$.
Moreover, $y_{k,m}\perp y_{k,m'}$ for all $k, m$ and $m'$ such that $0\le m,m'\le k, m\ne m'$.
Thus, arguing as in the proof of Theorem \ref{main1}, (ii)$\Rightarrow$(i),  one assumes that  $\la \not \in\si_{\pi e}(T),$
and arrives at a contradiction.

\smallskip

(i)$\Rightarrow$(ii):
Let $\e>0$ and $n> \max (4\|T\|^2\e^{-2},1).$ Fix $u\in H$ with $\|u\|=1$
and set $\la=e^{2\pi i/n}$. Choose
$$\e' \in \Bigl(0, \frac{1}{\sqrt{n}}\Bigl(\e-\frac{2\|T\|}{\sqrt{n}}\Bigr) \Bigr),$$
and let $u_0=u.$ Using  the fact that $\si_{\pi e}(T)\supset\TT,$  choose inductively orthonormal vectors $u_1,\dots,u_{n-1}$
such that $u_k\perp u$ and $\|(T-\la^k)u_k\|<\e'$ for $k=1,\dots,n-1$.

For $j=0,\dots,n-1$ set
$$
w_j=\frac{1}{\sqrt{n}}\sum_{k=0}^{n-1}\la^{jk}u_k.
$$

Then the vectors $w_0,\dots,w_{n-1}$ are orthonormal and
$$\frac{1}{\sqrt{n}}\sum_{j=0}^{n-1}w_j=u.$$

For $0\le j\le n-2$ we have
\begin{align*}
\|Tw_j-w_{j+1}\|&\le
\frac{1}{\sqrt{n}}\Bigl(\|Tu_0-u_0\|+\sum_{k=1}^{n-1}
\|\la^{jk} (Tu_k-\la^k u_k)\|\Bigr)\\
&\le
\frac{1}{\sqrt{n}}(1+\|T\|+n\e')
<\e.
\end{align*}
Similarly, $\|Tw_{n-1}-w_0\|<\e$.
\end{proof}
In view of \eqref{inclusion} the following corollary of  Theorem \ref{connes} is immediate.
\begin{corollary}\label{rohlin}
Let $T\in B(H)$ be such that $r(T)\le 1$. Then the following statements are equivalent.

(i) $\TT\subset\si(T).$

(ii) for all $\e>0$, $n> \max (4\|T\|^2\e^{-2},1)$ and $u\in H$, $\|u\|=1$ there exist orthonormal vectors
$w_0,\dots,w_{n-1}\in H$ such that
\begin{align*}
\|Tw_j-w_{j+1}\|<\e, \qquad 0\le j\le n-2, \qquad
\|Tw_{n-1}-w_0\|<\e
\end{align*}
and
$$
\frac{1}{\sqrt n}\sum_{j=0}^{n-1}w_j=u.
$$
\end{corollary}

\begin{remark}
Note that if we do not require the property $\frac{1}{\sqrt n}\sum_{j=0}^{n-1}w_j=u$ then there is an alternative direct approach to the construction of vectors $w_j$ as above. If  $T$ is unitary, $\e>0,$  $n>\max (2\pi\e^{-1},1),$ then for every  $w_0\in H$, $\|w_0\|=1,$ we construct vectors $w_1,\dots, w_{n-1}$ such that
\begin{align*}
\|T w_j-w_{j+1}\|<\e, \qquad j=0,1,\dots,n-2, \qquad
\|T w_{n-1}-w_0\|<\e,
\end{align*}
and, moreover, $$\sum_{j=0}^{n-1} w_j=0.$$

 Without loss of generality we may assume that $T$ is cyclic. Thus we may assume that $H=L^2(\mathbb T, \nu)$ for some probability measure $\nu$  and $T \in B(H)$ is defined by
 $$(Tx)(z)=zx(z), \qquad x\in L^2(\mathbb T, \nu), \quad z\in\TT \quad \text{a.e.}$$

For $z=e^{is}$ with $s\in[0,2\pi)$ and $n \in \mathbb N, n \ge 2,$ fix $t(z)\in[0,2\pi/n]$ such that
$s+t(z)\in\{\frac{2\pi}{n},\frac{4\pi}{n},\dots,\frac{2\pi (n-1)}{n}\}$ \ (note that $0$ and $2\pi$ are not elements of this set).
Clearly the function $z\mapsto t(z)$ is measurable (it has only a finite number of discontinuity points).

For $j=1,\dots,n$ define $w_j(z)=w_0(z)\cdot z^j e^{ijt(z)}, \,\, z\in\TT$. %Clearly $w_j\in H$, $\|w_j\|=1$ and $w_n=w_0$.
The family $\{ w_j: 0 \le j \le n-1\}$ is  orthonormal in $H,$ and $w_n=w_0$.
For $j=0,\dots,n-1$ we have
\begin{align*}
\|Tw_j-w_{j+1}\|^2&=
\int_\TT |w_0(z) z^{j+1}(e^{ijt(z)}-e^{i(j+1)t(z)})|^2 d\nu(z)\\
&\le
\int_\TT |w_0(z)|^2\cdot |1-e^{it(z)}|^2\, d\nu(z) \le
\Bigl(\frac{2\pi}{n}\Bigr)^2.
\end{align*}
So $$\|Tw_j-w_{j+1}\|\le \frac{2\pi}{n}<\e.$$

If $z=e^{is}, s\in[0,2\pi),$ then
$$\sum_{j=0}^{n-1} w_j(z)=
\sum_{j=0}^{n-1} w_0(z)e^{ijs}e^{ijt(z)} =0.
$$
\end{remark}

\section{Asymptotics for compressions of powers}

In this section we apply our numerical ranges ideology to study of
the interplay between the circle structure of the spectrum and the asymptotic properties  of
orbits. It appears that the property $\sigma (T)\supset \mathbb T$ strengthens the convergence of powers
of  $T \in B(H)$  on an appropriate, ``large'' subspace of $H.$
The property allows one to pass from the convergence of $(T^n)$  to zero in the weak operator topology
to the convergence to zero in the uniform operator topology of compressions of $(T^n)$ to an infinite-dimensional subspace.
At the same subspace, the norms of the compressions are as small as we please.

First, we will need the following numerical ranges lemma.
\begin{lemma}\label{lemmahamdan}
Let $T\in B(H)$ be such that $T^n\to 0$ in the weak operator topology. Suppose that for all $n\in \NN$,
\begin{equation}\label{we}
(0,\dots,0)\in W_e(T,\dots,T^n).
\end{equation}
Let $A\subset H$ be a finite set, $\e>0$ and $M\subset H$ be a subspace of a finite codimension. Then there exists a unit vector $x\in M$ such that
\begin{align*}
\sup_{n\ge 1}|\langle T^nx,x\rangle|\le\e,\quad \sup_{n\ge 1}|\langle T^nx,a\rangle|\le\e, \quad \text{and} \quad \sup_{n\ge 1}&|\langle T^{*n}x,a\rangle|\le\e,
\end{align*}
for all $a\in A.$
\end{lemma}

\begin{proof}
Clearly $T$ is power bounded by the uniform boundedness principle. Let $K=\sup\{\|T^n\|:n=0,1,\dots\}$.
It is apparent that also $T^{*n}\to 0$ in the weak operator topology.
Without loss of generality we may assume that $\max\{\|a\|:a\in A\}\le 1.$

Choose $s\in\NN$ such that $s>16K^2\e^{-2}$, and set formally $n_0=0$.

Choose $u_1\in M$ with $\|u_1\|=1$ arbitrarily. Choose $n_1>n_0$ such that
\begin{align*}
|\langle T^n u_1,u_1\rangle|&<\frac{\e}{4s}, \qquad n\ge n_1,\\
|\langle T^n u_1,a\rangle|&<\frac{\e}{4s}, \qquad n\ge n_1, a\in A,
\end{align*}
and
$$
|\langle T^{*n} u_1,a\rangle|<\frac{\e}{4s},\qquad n\ge n_1, a\in A.
$$

We construct unit vectors $u_2,\dots,u_s\in M$ in the following way:
Let $1\le r\le s-1$ and suppose that the unit vectors $u_1,\dots,u_r\in M$ and nonnegative integers $n_0<n_1<\cdots<n_r$ have already been constructed.

By assumption \eqref{we}, there exists a unit vector $u_{r+1}\in M$ such that
$$
u_{r+1}\perp\{T^nu_k, T^{*n}u_k,T^na,T^{*n}a:\, 0\le n\le n_r, 1\le k\le r, a\in A\}
$$
and
$$
|\langle T^nu_{r+1},u_{r+1}\rangle|< \frac{\e}{4}, \qquad 1\le n\le n_r.
$$
Find $n_{r+1}>n_r$ such that
$$
|\langle T^nu_{r+1},u_k\rangle|<\frac{\e}{4s},
$$
$$
|\langle T^{*n}u_{r+1},u_k\rangle|<\frac{\e}{4s},
$$
$$
|\langle T^nu_{r+1},a\rangle|<\frac{\e}{4s},
$$
and
$$
|\langle T^{*n}u_{r+1},a\rangle|<\frac{\e}{4s}
$$
for all $n\ge n_{r+1}$, $1\le k\le r+1$ and $a\in A$.

Let $u_1,\dots,u_s$ and $n_0,\dots,n_s$ be constructed in this way.
Set $$x=\frac{1}{\sqrt{s}}\sum_{k=1}^s u_k.$$
Clearly $x\in M.$ Moreover, $\|x\|=1$ since the vectors $u_k, 1 \le k \le s,$ are orthonormal.

For $n\ge n_s$ we have
$$
|\langle T^nx,x\rangle|\le
s^{-1}\sum_{k,k'=1}^s|\langle T^nu_k,u_{k'}\rangle|\le
s^{-1}s^2\frac{\e}{4s}<\e.
$$

Let $0\le r\le s-1$ and $n_r<n\le n_{r+1}$. Then
$$
|\langle T^nx,x\rangle|=
s^{-1}\Bigl|\sum_{k,k'=1}^s\langle T^nu_k,u_{k'}\rangle\Bigr|
$$
$$
\le
s^{-1}\sum_{k,k'=1}^{r}|\langle T^nu_k,u_{k'}\rangle| +
s^{-1}\sum_{k=1}^{r+1}|\langle T^nu_{r+1},u_{k}\rangle|+
s^{-1}\sum_{k=1}^{r}|\langle T^nu_k,u_{r+1}\rangle|
$$
$$
+
s^{-1}\sum_{k=r+2}^{s}|\langle T^nu_k,u_{k}\rangle|+
s^{-1}\sum_{1\le k,k'\le s, k\ne k'\atop \max\{k,k'\}\ge r+2}|\langle T^nu_k,u_{k'}\rangle|,
$$
where the last term is equal to $0$ by the construction.
So
\begin{align*}
|\langle T^nx,x\rangle|\le&
s^{-1}r^2\frac{\e}{4s}+
s^{-1}\|T^nu_{r+1}\|\cdot\|\sum_{k=1}^{r+1}u_k\|\\
+&
s^{-1}\|T^{*n}u_{r+1}\|\cdot\|\sum_{k=1}^{r}u_k\|
+s^{-1}(s-r-1)\frac{\e}{4}\\
\le&
\frac{\e}{4}
+s^{-1}K\sqrt{r+1}+s^{-1}K\sqrt{r}+\frac{\e}{4}
\le \e.
\end{align*}
Hence $$\sup_{n\ge 1}|\langle T^nx,x\rangle|\le \e.$$

Let $a\in A$. For $n\ge n_{s}$ we have
$$
|\langle T^nx,a\rangle|\le
\frac{1}{\sqrt{s}}\sum_{k=1}^s|\langle T^nu_k,a\rangle|\le
\frac{1}{\sqrt{s}}\cdot s\cdot\frac{\e}{4s}<\e.
$$
Let $0\le r\le s-1$ and $n_r\le n<n_{r+1}$. Then
\begin{align*}
|\langle T^nx,a\rangle|&\le
\frac{1}{\sqrt{s}}\sum_{k=1}^r|\langle T^nu_k,a\rangle|
+\frac{1}{\sqrt{s}}|\langle T^nu_{r+1},a\rangle|
+\frac{1}{\sqrt{s}}\sum_{k=r+2}^s|\langle T^nu_k,a\rangle|\\
&\le
\frac{1}{\sqrt{s}}\cdot r\cdot\frac{\e}{4s}+
\frac{1}{\sqrt{s}}\cdot K
<\e.
\end{align*}
So $$\sup_{n\ge 1}|\langle T^nx,a\rangle|\le\e$$ for all $a\in A$.

The property $\sup_{n\ge 1}|\langle T^{*n}x,a\rangle|\le\e$ for all $a \in A$ can be proved similarly.
\end{proof}

Now we are ready to use essential numerical ranges for the study of operator norm convergence. The following theorem is one of the main results of the paper.

\begin{theorem}\label{theoremhamdangeneral}
Let $T\in B(H)$ and let $T^n\to 0$ in the weak operator topology. Suppose that $(0,\dots,0)\in W_e(T,\dots,T^n)$ for all $n\in\NN$.  Then for every $\e>0$ there exists an infinite-dimensional subspace $L$ of $H$ such that
$$
\sup_{n\ge 1}\|P_LT^nP_L\|\le \e \qquad\text{and}\qquad \lim_{n\to\infty}\|P_L T^n P_L\|=0,
$$
where $P_L$ is the orthogonal projection on $L.$
\end{theorem}

\begin{proof}
Let $n_0=1$. By Lemma \ref{lemmahamdan}, there exists $y_1\in H$, $\|y_1\|=1,$ such that
$$
\sup_{n\ge 1}|\langle T^ny_1,y_1\rangle|<\frac{\e}{2}.
$$
Find $n_1>n_0$ satisfying
$$
|\langle T^ny_1,y_1\rangle|<\frac{\e}{4}, \qquad n\ge n_1.
$$
We construct inductively unit vectors $y_2, y_3,\dots\in H$ and nonnegative integers $n_2<n_3<\cdots$ in the following way:

Let $r\in \mathbb N$ and suppose that $y_1,\dots, y_r$ and  $n_1,\dots,n_r$ have already been constructed.
By Lemma \ref{lemmahamdan}, there exists $y_{r+1}\in H$ such that
\begin{align*}
\|y_{r+1}\|&=1,\\
y_{r+1}&\perp\{T^ny_k, T^{*n}y_k: 0\le n<n_r, 1\le k\le r\},
\\
\sup_{n\ge 1} |\langle T^ny_{r+1},y_{r+1}\rangle|&<\frac{\e}{2^{r+3}(r+1)},
\\
\sup_{n\ge 1} |\langle T^ny_{r+1},y_{k}\rangle|&<\frac{\e}{2^{r+3}(r+1)}, \qquad 1\le k\le r,
\\
\sup_{n\ge 1} |\langle T^{*n}y_{r+1},y_{k}\rangle|&<\frac{\e}{2^{r+3}(r+1)}, \qquad 1\le k\le r.
\end{align*}
Find $n_{r+1}>n_r$ satisfying
$$
|\langle T^ny_k,y_{k'}\rangle|\le \frac{\e}{2^{r+4}(r+1)}, \qquad 1\le k,k'\le r+1, n\ge n_{r+1}.
$$

Now suppose that  $y_k$ and  $n_k$  have been constructed as above.
Let $$L=\bigvee_{r=1}^\infty y_r.$$ Clearly $L$ is an infinite-dimensional subspace with an orthonormal basis $(y_r).$

Let $y\in L$, $\|y\|=1$. Then
$$y=\sum_{k=1}^\infty \al_k y_k \qquad \text{with}\qquad  \sum_{k=1}^\infty|\al_k|^2=1.$$
Note that $\sum_{k=1}^r|\al_k|\le \sqrt r$ for all $r \in \mathbb N$.

Let $r\in \mathbb N\cup\{0\},$ $n\in\NN$, and $n_r\le  n< n_{r+1}$. Then
$$\langle T^ny_k,y_{k'}\rangle=0$$
 if $k\ne k'$ and $
\max\{k,k'\}\ge r+2.$
So
\begin{align*}
|\langle T^ny,y\rangle|\le&
\sum_{k,k'=1}^r|\al_k\bar\al_{k'}|\cdot|\langle T^ny_k,y_{k'}\rangle|+
\sum_{k=1}^{r+1}|\al_{r+1}\bar\al_k|\cdot |\langle T^ny_{r+1},y_{k}\rangle|\\
+&
\sum_{k=1}^r|\al_k\bar\al_{r+1}|\cdot |\langle T^ny_k,y_{r+1}\rangle|+
\sum_{k=r+2}^\infty|\al_k|^2\cdot |\langle T^ny_k,y_{k}\rangle|
\\
\le&
r\cdot\frac{\e}{2^{r+3}r}+
\sqrt{r+1}\cdot\frac{\e}{2^{r+3}(r+1)}+
\sqrt{r}\cdot\frac{\e}{2^{r+3}(r+1)}+
\frac{\e}{2^{r+4}(r+2)}\\
<&\frac{\e}{2^{r+1}}.
\end{align*}
Thus, if $n_r\le n\le n_{r+1}$ then the numerical radius $w(P_LT^nP_L)$ of $P_LT^nP_L$ satisfies
$$w(P_LT^nP_L):=\sup\bigl\{|\langle P_LT^nP_L y,y\rangle|: y\in H, \|y\|=1\bigr\}\le 2^{-r-1}\e.$$
Since for any $T \in B(H),$ one has $\|T\|\le 2 w(T)$ (see e.g. \cite[p. 33]{Halmos} or \cite[Theorem 1.3-1]{Gustafson}),
we infer that
$$\|P_LT^nP_L\|\le 2^{-r}\e.$$
Hence $$\sup_{n\ge 1}\|P_LT^nP_L\|\le\e \quad \text{and} \quad \lim_{n\to\infty}\|P_LT^nP_L\|=0.$$
\end{proof}

The next corollary of Theorem \ref{theoremhamdangeneral} replaces the numerical ranges condition $(0,\dots,0)\in W_e(T,\dots,T^n), n\in\NN,$
taking into account \emph{all} powers of $T,$
by the more transparent spectral assumption $0\in\widehat\sigma_e(T).$

\begin{corollary}\label{hamdancorollarygeneral}
Let $T\in B(H)$,
%${{\rm dim}\, H}=\infty,$
and let $T^n\to 0$ in the weak operator topology. Suppose that $0\in\widehat\sigma_e(T).$ Then for every  $\e>0$ there exists an infinite-dimensional subspace $L$ of $H$ such that
$$
\sup_{n\ge 1}\|P_LT^nP_L\|\le \e \qquad\hbox{and}\qquad \lim_{n\to\infty}\|P_LT^nP_L\|=0,
$$
where $P_L$ is the orthogonal projection on $L.$
In particular, this is true if the assumption $0\in\widehat\sigma_e(T)$ is replaced by $\mathbb T \subset \sigma (T).$
\end{corollary}

\begin{proof}
We consider two cases. If $0 \in {\rm Int} \, \widehat\sigma_e(T)$ then by Theorem  \ref{lambdawe} we have  $0\in W_e(T, \dots, T^n)$ for every $n \in \mathbb N,$ and the statement follows from Theorem \ref{theoremhamdangeneral}.

 On the other hand, if
$0 \in \partial \widehat\sigma_e(T),$ then using elementary properties of polynomial convex hulls and \cite[Proposition III.19.1]{Muller}, we have
$$
0 \in \partial \widehat\sigma_e(T) \subset \partial \sigma_e(T) \subset \sigma_{\pi e}(T),
$$
so that
$$
(0, \dots, 0) \in \sigma_{\pi e} (T, \dots, T^n)\subset W_e(T, \dots, T^n),
$$
for every $n \in \mathbb N.$
(Alternatively,
using \cite[Proposition III.19.1 and Corollary III.19.16]{Muller}, one may note that
$$
0 \in \partial \widehat\sigma_e(T) = \partial \widehat \sigma_{\pi e} (T) \subset \sigma_{\pi e}(T), \qquad n \in \mathbb N,
$$
and then $(0, \dots, 0) \in W_e(T, \dots, T^n)$ as above.)
Thus $$(0,\dots,0)\in W_e(T,\dots,T^n), \qquad   n \in \mathbb N,$$ again, and we can use Theorem \ref{theoremhamdangeneral}.

If $\mathbb T \subset \sigma (T),$ then \eqref{inclusion}  yields
$\mathbb T \subset \sigma_e (T),$ so that $0 \in\widehat\sigma_e(T).$
\end{proof}
\begin{remark}
Observe that one can replace the assumption $0\in\widehat\sigma_e(T)$ in Corollary \ref{hamdancorollarygeneral} by $0 \in \Int \widehat \sigma(T).$
\end{remark}
If $T$ is unitary then the above corollary can be sharpened.
The result below is an essential generalization of the main result in \cite{Hamdan} (with a completely different proof).

\begin{corollary}\label{hamdancorollaryunitary}
%Let ${{\rm dim}\, H}=\infty,$ and
Let $T$ be a unitary operator on $H$ such that  $T^n\to 0$ in the weak operator topology.
Then the following
conditions are equivalent.

(i) $\si(T)=\TT.$

(ii) for every $\e>0$ there exists $x\in H$, $\|x\|=1,$ with $$\sup_{n\ge 1}|\langle T^nx,x\rangle|<\e.$$

(iii) for every $\e>0$ there exists an infinite-dimensional subspace $L\subset H$ such that
$$
\sup_{n\ge 1}\|P_L T^nP_L\|\le \e \qquad\hbox{and}\qquad \lim_{n\to\infty}\|P_L T^nP_L\|=0,
$$
where $P_L$ is the orthogonal projection on $L.$
\end{corollary}

\begin{proof}
The implication (ii)$\Rightarrow$(i) follows from Theorem \ref{main1}.
If $\si(T)=\TT$ then $\si_{\pi e}(T)=\TT$ by \eqref{inclusion}. So (i)$\Rightarrow$(iii) follows from the previous corollary. The implication (iii)$\Rightarrow$(ii) is trivial.
\end{proof}
\section{Acknowledgements}
The authors would like to thank the referee for helpful comments and remarks
that essentially streamlined our original arguments used in Section \ref{numersection},
 and significantly improved our presentation.

\end{document}